\documentclass[11pt]{amsart}

\usepackage{amsmath,amsfonts,amssymb,amscd}
\usepackage{mathrsfs}
\usepackage{enumerate}
\usepackage[utf8]{inputenc}
\usepackage[T1]{fontenc}
\usepackage{latexsym}
\usepackage{graphics, graphicx,color}
\usepackage{verbatim}
\numberwithin{equation}{section}

\newtheorem{teo}{Theorem}[section]

\newtheorem{lema}[teo]{Lemma}
\newtheorem{cor}[teo]{Corollary}
\newtheorem{afir}[teo]{Claim}
\newtheorem{de}[teo]{Definition}

\newtheorem{pr}[teo]{Proposition}
\newtheorem{re}[teo]{Remark}

\newtheorem{q}{Question}
\newtheorem{bigteo}{Theorem}

\author[]{Felipe Nobili}

\title[]{Minimality of one invariant lamination for partially hyperbolic attractors}

\begin{document} 
\begin{abstract} We prove that at least one of the two invariant laminations of a strongly partially hyperbolic attractor with one-dimensional center bundle is minimal. This result extends   those in \cite{B2,B10} about minimal foliations for robustly transitive diffeomorphisms. 

\end{abstract}

\maketitle
\tableofcontents

\newpage
\section{Introduction}

In Hyperbolic Theory, a transitive attractor holds many robust properties such as being a homoclinic class. In particular, its invariant manifolds intersect the attractor in a dense subset of it. Given that non-hyperbolic sets  appears in a relevant subset of dynamical systems (open sets of $\operatorname{Diff}^1(M)$), it is natural to investigate whether these robust properties survives in the non-hyperbolic settings. \smallskip

In this work we deals with transitive attractors that present a weaker form of hyperbolicity known as \emph{partial hyperbolicity}.  In this situation we also have invariant submanifolds associated to the expanding and contracting bundles. As in the hyperbolic case, we can wonder if these submanifolds intersect the attractor in a dense subset.\smallskip

A simpler case to start with is when the center bundle is one-dimensional. This gives two properties of the attractor that play an important role in our study. First, this guarantees that the attractor is far from homoclinic tangencies. Secondly,  it gives rise to invariant central curves for the hyperbolic periodic points of the attractor (see Section~\ref{ccpp}). \smallskip

 In \cite{B2}, the case of 3-dimensional robustly transitive diffeomorphisms was investigated. It was proved that the leaves of at least one of the invariant foliations, those that integrate the extremal bundles of the partial hyperbolic splitting, is dense in the ambient manifold. Later, this result was extended to higher dimensions in \cite{B10}.\smallskip

Our main result translates \cite{B2,B10} to the context of robustly transitive attractors. Let us state our results in a more precise way.\smallskip

\section{Statemant of the results}

Let $\operatorname{Diff}^1(M)$ denote the space of $C^1$ diffeomorphisms from a compact Riemannian manifold $M$ to itself, endowed with the usual uniform $C^1$-topology. Consider a diffeomorphism  $f \in \operatorname{Diff}^1(M)$ and an $f$-invariant set $\Lambda$ over whom the tangent bundle  admits a partially hyperbolic splitting $T_\Lambda M = E^{s}\oplus E^c\oplus E^u$. This means that the extremal subbundles $E^s$ and $E^u$ are, respectively, uniformly contracting and uniformly expanding by the action of the derivative $Df$ of $f$, and that $E^c$ has an intermadiate behaviour\footnote{We refer to Appendix B of \cite{B3} for a precise definition and a list of elementary properties of partially hyperbolic systems.}. If both $E^s$ and $E^u$ are nontrivial, we say that $\Lambda$ is a \emph{strongly partially hyperbolic set}. When the whole manifold $M$ is strongly partially hyperbolic, we call
$f$ a \emph{strongly partially hyperbolic diffeomorphism}.\smallskip

In what follows we assume that the dimensions of $E^s(x)$, $E^c(x)$, and $E^u(x)$ do not depend on the point $x \in \Lambda$. Due to the $Df$-invariance and the continuity of the splitting, this assumption is automatically satisfied when $\Lambda$ is transitive.\smallskip

According to \cite{B1}, strong partial hyperbolicity leads to the existence of two laminations\footnote{See Section 6.2 of \cite{B15} and  Appendix IV of \cite{B28} for a more detailed account on this subject.}  $\mathcal{F}^s(f)$ and $\mathcal{F}^u(f)$ on the set $\Lambda$, named \emph{strong stable} and \emph{strong unstable} laminations, respectively. The leaves of $\mathcal{F}^s(f)$ and $\mathcal{F}^u(f)$ are dynamically defined immersed submanifolds of $M$ tangent to the bundles $E^s$ and $E^u$, respectively, at each point intersecting $\Lambda$. We denote by $\mathcal{F}^s(x,f)$ and $\mathcal{F}^u(x,f)$ the leaves of these laminations passing through the point $x \in \Lambda$. When there is no risk of misunderstanding, we omit $f$ in the above notations and just write $\mathcal{F}^s$, $\mathcal{F}^u$, $\mathcal{F}^s(x)$, and $\mathcal{F}^u(x)$ instead. 
\smallskip

When $\Lambda =M$, the laminations are commonly referred to as \emph{foliations}. We say that the foliation is \emph{minimal} if the orbit of each leaf by $f$ is a dense subset of $M$. Assuming that $M$ is connected, being a minimal foliation actually means that each leaf itself is dense in $M$ (see Lemma 4.5 of \cite{B2}).
When the strong stable (resp. unstable) foliation of a partially hyperbolic diffeomorphism is minimal we speak of $s$-minimality (resp. $u$-minimality).

\smallskip

Let us denote by $\mathrm{RTPH}_1(M)$ the open subset of $\operatorname{Diff}^1(M)$ consisting of robustly transitive, robustly non-hyperbolic, partially hyperbolic diffeomorphisms with one-dimensional center bundle. \smallskip

Our results is motivated by the following theorem  about minimal foliations.

\begin{teo}[\cite{B2},\cite{B10}] \label{t.1} There is an open and dense subset of $\mathrm{RTPH}_1(M)$ consisting of diffeomorphisms which are either robustly $s$-minimal or robsutly $u$-minimal.\end{teo}

\smallskip

In our study, we need a suitable notion of $u$- and $s$-minimality that fit to the context of partially hyperbolic proper subsets of $M$. A first idea is to require that the orbit of each leaf accumulates over the whole set $\Lambda$. In principle, this accumulation could be done ``outside'' $\Lambda$ (that is, with a small or none intersection with $\Lambda$). However, a more interesting property would be that the orbit of each leaf intersect $\Lambda$ in a dense subset. Unlike the case of transitive diffeomorphisms, this requirement do not implies that each leaf itself has a dense intersection with the set. Nevertheless, only a finite number of iterates on each leaf suffices to get the desired dense intersection. Altogether, this considerations leads to the definition of $u$- and $s$-minimality we deal with in this paper (see Section~\ref{minimality}). \smallskip

Fixed an open set $U \subset M$, denote by $\mathrm{RTPHA}_1(U)$ (resp. $\mathrm{GTPHA}_1(U)$) the subset of $\operatorname{Diff}^{1}(M)$ of diffeomorphisms $f$ for which the maximal $f$-invariant subset $\Lambda_f(U)$ of $U$ is a robustly non-hyperbolic, partially hyperbolic set with one-dimensional center bundle that is robustly (resp. generically) transitive. The next two theorems summarizes the main results in this paper.

\begin{bigteo}[generically transitive case] For every open subset $U \subset M$, there is a residual subset of $\mathrm{GTPHA}_1(U)$ 
consisting of diffeomorphisms $g$ for which $\Lambda_g(U)$ is either
generically $s$-minimal or generically $u$-minimal. \end{bigteo}

\begin{bigteo}[robustly transitive case] For every open subset $U \subset M$, there is an open and dense  subset of $\mathrm{RTPHA}_1(U)$ 
consisting of diffeomorphisms $g$ for which $\Lambda_g(U)$ is either
robustly  $s$-minimal or robustly $u$-minimal. \end{bigteo}

Note that Theorem \ref{t.1} is a particular case of Theorem B when $U=M$.\medskip

The main difficulty in adapting the global Theorem~\ref{t.1} to our local case is that the partial hyperbolicity is defined only in the attractor. For instance, we can not saturate a strong stable leaf with unstable ones to obtain a co-dimension one topological manifold, since not all the points in the stable leaf belongs to the attractor. A more technical constraint is that a co-dimension one manifold may not divide the set locally into two nonempty components\footnote{This is the reason why oriantability assumptions is not as much helpful in our case as it was for transitive diffeomorphisms in \cite{B2}}.

\section{Preliminaries} \label{pre}

\smallskip

We introduce in this section the basic definitions and terminology that we use throughout this paper. 
\smallskip

Fix  $f \in \operatorname{Diff}^{1}(M)$. Given a subset $X \subset M$, we denote the orbit, the forward orbit, and the backward orbit of $X$ by  $\mathcal{O}_f(X)$, $\mathcal{O}^+_f(X)$, and $\mathcal{O}^-_f(X)$, respectively. For every open subset $U$ of $M$, we define the maximal $f$-~invariant set of $f$ in $U$ by 
$$
\Lambda_f(U):=\bigcap_{n \in \mathbb{Z}} f^n(U).
$$ 

With respect to $f$, a compact invariant set $\Lambda \subset M$ is said:
 
\medskip

\begin{itemize}

\item \emph{Isolated} or \emph{locally maximal}: If there is an open neighborhood $U$ of $\Lambda$ such that $\Lambda = \Lambda_f(U)$. Equivalently, $\Lambda$ is the maximal invariant subset of $f$ in $U$. Any  open neighborhood $U$ of $\Lambda$ satisfying $\Lambda = \Lambda_f(U)$ is called an \emph{isolating block} of $\Lambda$.

\smallskip

\item An \emph{attractor}:  If there is an open neighborhood $U$ of $\Lambda$ such that $f(\overline{U}) \subset U$ and $\Lambda = \bigcap_{n \in \mathbb{N}} f^n(U)$. We call $\Lambda$ a \emph{proper} attractor if $U \not= M$ , and thus $\Lambda \not= M$. 

\smallskip

\item \emph{Transitive}: If there is $ x \in \Lambda$ such that its forward orbit $\mathcal{O}^+_f(x)$ is dense in $\Lambda$. In our setting, this is equivalent to the following property: Given any pair $V_1,V_2$ of (relative) nonempty open sets of $\Lambda$, there is $n \in \mathbb{Z}$ such that $f^n(V_1) \cap V_2 \not= \emptyset$.

\item \emph{Robustly transitive set (resp. attractor)}: If there are an isolating block $U$ of $\Lambda$ and a neighborhood $\mathcal{U}$ of $f$ such that, for every $g \in \mathcal{U}$, the set $\Lambda_g(U)$ is a compact transitive set (resp. attractor) with respect to $g$.

\smallskip

\item \emph{Generically transitive set (resp. attractor)}: If there are an isolating block $U$ of $\Lambda$, a neighborhood $\mathcal{U}$ of $f$, and a residual subset $\mathcal{R} \subset \mathcal{U}$ such that, for every $g \in \mathcal{R}$, the set $\Lambda_g(U)$ is a compact transitive set (resp. attractor) with respect to $g$. 
\end{itemize}

\smallskip

\begin{re} \label{cont.isolated} \em{Isolated sets vary, \emph{a priori}, just upper semicontinuously. By an abuse of terminology, we call the set $\Lambda_g(U)$ the \emph{continuation} of the set $\Lambda_f(U)$ when $g$ varies in a small neighborhood of~$f$. } 
\end{re} 
 
\begin{re} \label{r.continuation} 
 \em{An attractor $\Lambda$ of a diffeomorphism $f$ is an isolated set, so we also denote it by $\Lambda_f(U)$ for some isolating block $U$ of $\Lambda$. Observe that if $g$ is close enough to $f$, then the  continuation $\Lambda_g(U)$ of $\Lambda_f(U)$ is also an attractor for $g$. Clearly, if $\Lambda_f(U)$ is a proper set, so is the continuation $\Lambda_g(U)$}.
\end{re}

\begin{re} {\em{In the definition of attractors, some authors requires the additional property that $\Lambda_f(U)$ is transitive. We do not follow this convention. The reason is that we want to talk about the continuation of the attractor as in Remark~\ref{r.continuation}, and transitivity is not a robust property in general.}} 
\end{re}

Let $\Lambda$ be a partially hyperbolic set of a $C^1$ diffeomorphism $f$. Given a hyperbolic periodic point $p \in \Lambda$, we denote the period of $p$ by $\pi(p)$. If $g \in \operatorname{Diff}^{1}(M)$ is sufficiently close  to $f$, then there is a hyperbolic continuation of $p$ for $g$ that we denote by $p_g$, which satisfies $\pi(p_g) = \pi(p)$. \\

Given $r>0$, we denote by $\mathcal{F}^{s}_r (x,f)$ the open ball of radius $r$ centered at $x$, relative to the induced distance on the leaf $\mathcal{F}^{s}(x,f)$. \smallskip 

\begin{re} \label{r.continuity} \em{The leaves of $\mathcal{F}^s(g)$ and $\mathcal{F}^u(g)$ depend continuously on the diffeomorphism $g$. This means that, fixed $r>0$ and $\varepsilon > 0$, if $g$ is sufficiently close to $f$ and $x \in \Lambda_f(U)$ is sufficiently close to $y \in \Lambda_g(U)$, then the disk $\mathcal{F}^s_r(y,g)$ is $\varepsilon$-close to the disk  $\mathcal{F}^s_r(x,f)$ with respect to the $C^1$-topology. }
\end{re}

 The local and global stable manifolds of $p$ are denoted by $W^s_r(p,f)$ and $W^s(p,f)$, respectively. Also, $W^s_{r}(\mathcal{O}_f(p),f)$ and $W^s(\mathcal{O}_f(p),f)$ stand for the global and local stable manifolds of the orbit of $p$. The dimension of $W^s(p,f)$ as a submanifold of $M$ is called the \emph{index} of $p$ and is denoted by $\operatorname{index}(p)$. If $g$ is close to $f$, then $\operatorname{index}(p_g) = \operatorname{index}(p)$.

 Similar notations are considered to the unstable manifold and lamination. As before, when there is no risk of misunderstanding, we omit $f$ in these notations.

\begin{re}\em{We denote the dimensions of the bundles $E^s$, $E^c$ and $E^u$ of the partially hyperbolic splitting of $\Lambda$ by $d^s$, $d^c$, and $d^u$, respectively.
Clearly, $d^s \leq \operatorname{index}(p) \leq d^s + d^c$. In particular, when $d^c=1$, there are only two possibilities for $\operatorname{index}(p)$, which are $d^s$ or $d^s +1$.}
\end{re}

The set of all periodic points in $\Lambda$ with index $d^s$ and $d^s + 1$ are denoted by $\operatorname{Per}_{d^s}(f_{|_{\Lambda}})$ and $\operatorname{Per}_{d^s+1}(f_{|_{\Lambda}})$, respectively.  \smallskip

When we deal with perturbations of a diffeomorphism $f$, we usually want that the new diffeomorphism is so close to $f$ so that it inherit the robust properties of $f$. Hence, we introduce the following definition.   

\begin{de} \label{d.compatible}
\em{Let $\Lambda$ be an isolated set of a diffeomorphism $f$ and $U \subset M$ be an isolated block of $\Lambda$.  We say that a neighborhood $\mathcal{U}$ of $f$ is \emph{compatible} (with respect to $U$) if it is sufficiently small so that, for all $g \in \mathcal{U}$, we have:

\begin{itemize}

\item the set $\Lambda_g(U)$ is an isolated set; 

\item if $\Lambda_f(U)$ is an attractor of $f$, then $\Lambda_g(U)$ is an attractor of $g$;

\item if $\Lambda_f(U)$ is a partially hyperbolic set, then $\Lambda_g(U)$ is a partially hyperbolic set of $g$ with the same bundle dimensions $d^s$, $d^c$, and $d^u$; 

\item if $\Lambda_f(U)$ is a generically (resp. robustly) transitive set of $f$, then $\Lambda_g(U)$ is a generically (resp. robustly) transitive set of $g$.

\end{itemize} }

\end{de}

\section{$C^1$-Generic dynamics}
 \label{genericsection}
 
\medskip
 
 In this section we gather some generic properties of diffeomorphisms in $\operatorname{Diff}^{1}(M)$. 
We say that a property $\mathbf{P}$ is \emph{$C^1$-generic}  if $\mathbf{P}$ holds for every diffeomorphism in a residual ($G_{\delta}$ and dense) subset $\mathcal{R}$ of $\operatorname{Diff}^{1}(M)$. 

\smallskip

\subsection{$C^1$-Generic homoclinic classes}

\begin{de}[Homoclinic class]
\label{d.homoclinic}
{\em{
Let $p $ be a hyperbolic periodic point of a diffeomorphism $f$. 
A \emph{homoclinic point} $x$ of $p$ is a point whose forward and backward iterates converge to the orbit $\mathcal{O}_f(p)$ of $p$ 
(i.e., $x \in W^s(\mathcal{O}_f(p)) \cap W^u(\mathcal{O}_f(p))$). 
If the stable and unstable manifolds of the orbit of $p$ meet transversely at $x$, we say that $x$ is a \emph{transverse homoclinic point}. Otherwise, we say that $x$ is a \emph{homoclinic tangency}.

\medskip

The homoclinic class of $p$, denoted by
$H(p,f)$,
is the closure of the set of all transverse homoclinic points of $p$. That is, 
$$H(p,f)= \overline{W^s(\mathcal{O}_f(p)) \pitchfork W^u(\mathcal{O}_f(p))}.$$}}
\end{de}

\begin{re} \label{r.HC.properties}
 {\em{Homoclinic classes are transitive sets and contain a dense subset of periodic points. By the persistence of the transverse intersections between the invariant manifolds of $p$, we find that homoclinic classes vary lower semicontinuously. See  Chapter 10.4 of \cite{B3} for a detailed discussion about homoclinic classes. }}
\end{re}


\begin{teo}[\cite{B8,B13}] \label{t.generic}
There is a residual subset $\mathcal{R}_0$ of $\operatorname{Diff}^{1}(M)$ such that, for every $f \in \mathcal{R}_0$, the following holds:

\begin{enumerate}

\item 
The diffeomorphism $f$ is Kupka-Smale: every periodic point of $f$ is hyperbolic and their invariant manifolds met transversely.

\item The set $\operatorname{Per}(f)$ of periodic points of $f$ is  dense in the nonwandering set $\Omega(f)$ of $f$. In particular, any isolated transitive compact set has a dense subset of periodic points. 



\item
Any transitive set intersecting a homoclinic
class is contained in it. In particular, 
any pair of homoclinic classes are either disjoint or coincide.

\end{enumerate}
\end{teo}

Items (1) and (2) are the main theorem in \cite{B13}. Items (3) is item 3 of Theorem A in \cite{B8}. 
\smallskip

\begin{re} \label{5gta} \em{By item (3), a generic transitive attractor coincide with the homoclinic classes of its periodic points.}  \end{re}

\subsection{$C^1$-Generic transitive sets and attractors} \
\medskip
 
In this subsection we gather some useful properties of transitive sets and attractors in the $C^1$-generic setting.

\begin{pr}
\label{Ab} 
There is a residual subset $\mathcal{R}_1$ of $\operatorname{Diff}^{1}(M)$ such that, if 
$f \in \mathcal{R}_1$ and $\Lambda_f(U)$ 
is an isolated subset of $M$, then the following  hold:
\begin{enumerate}
\item
If $\Lambda_f(U)$ is a transitive attractor, then there is a neighborhood $\mathcal{U}$ of $f$ such that, 
for every $g \in \mathcal{R}_1 \cap \mathcal{U}$, the set $\Lambda_g(U)$ 
is a transitive attractor. 
In other words, the set $\Lambda_f(U)$ is a generically transitive attractor. 

\item
If $\Lambda_f(U)$ is non-hyperbolic, then it 
contains a pair of (hyperbolic) saddles of different indices.
\end{enumerate}
\end{pr}

Item (1) is Theorem B of \cite{B4}. Item (2) is due to Man\~e in the proof of the Ergodic Closing Lemma \cite{B16}. \\

Next proposition is a translation of some results in \cite{B2} about generic transitive diffeomorphisms to the context of generic transitive attractors. The same arguments in \cite{B2} can be applied to our context without substantial amendments, so here we only sketch the proof.

\begin{pr}\label{Bld}  
There is a residual subset $\mathcal{R}_2$ of 
$\operatorname{Diff}^{1}(M)$, with $\mathcal{R}_2 \subset \mathcal{R}_1$, consisting of diffeomorphisms
$f$ satisfying the following property. Let $f \in \mathcal{R}_2$ and  $\Lambda_f(U)$ be 
a transitive isolated set of $f$ that is  
partially hyperbolic with one-dimensional center bundle.
Then, for every pair of hyperbolic periodic points 
$p,q \in \Lambda_f(U)$ with indices $d^s$ and $d^s+1$ 
respectively,
there is an open set $\mathcal{V}_{p,q} \subset \operatorname{Diff}^{1}(M)$, 
with $f \in \overline{\mathcal{V}_{p,q}}$, 
such that, for all $g \in \mathcal{V}_{p,q}$ it hold:
\begin{enumerate}[1)]
\item $W^s(\mathcal{O}_g(q_g)) \subset \overline{W^s(\mathcal{O}_g(p_g))}$
and
$W^u(\mathcal{O}_g(p_g)) \subset \overline{W^u(\mathcal{O}_g(q_g))}$.

\item
if $\Lambda_g(U)$ is transitive, then
 $\Lambda_g(U) \subset \overline{W^s(\mathcal{O}_g(p_g))} 
\cap \overline{W^u(\mathcal{O}_g(q_g))}$.

\item if $\Lambda_f(U)$ is robustly transitive, then $\Lambda_g(U) \subset H(p_g,g)$.

\end{enumerate}
\end{pr}

\begin{proof}[sketch of the proof] By item (2) of Proposition \ref{Ab}, the set $\Lambda_f(U)$ is generically transitive and contains saddles $p$ and $q$ of indices $d^s$ and $d^s+1$, respectively. Observe that for every $g$ close to $f$ one has $p_g,q_g\in \Lambda_g(U)$. By Proposition 1.1 of \cite{B14}, we can assume that there is an \emph{heterodimensional cycle}\footnote{See \cite{B7} for the definition and a general overview on this topic.} associated to $p_g$ and $q_g$ for every $g$ in a dense subset of a neighbourhood of $f$. Since $E^c$  is one-dimensional, there is no homoclinic tangencies associated to the periodic points of any $g$ in a neighbourhood of $f$. These shows that we are in the hypotheses of Proposition 2.6 in \cite{B2}, proving item 1). The proof of item 2) goes the same as the proof of Propositions 2.1 and 2.3 of \cite{B2}. The idea is that the invariant manifolds with bigger dimensions ($W^s(\mathcal{O}_g(p_g))$ and $W^u(\mathcal{O}_g(q_g))$) accumulate at every point in $\Lambda_f(U)$. Using item 1), we see that the same hold for the manifolds with lower dimensions. Item 3) is the equivalent of Corollary 2.5 in \cite{B2} for proper attractors.\end{proof}

\begin{cor} \label{CCCCC} There is an open and dense subset $\mathcal{B}$ of $\mathrm{RTPHA}_1(U)$ such that, for every $f \in \mathcal{B}$, the attractor  $\Lambda_f(U)$ is a homoclinic class. Consequently, the attractors $\Lambda_f(U)$ depend continuously on $f \in \mathcal{B}$. 
\end{cor}

\begin{proof} Applying Proposition~\ref{Bld} to a dense subset of diffeomorphisms in $\mathrm{RTPHA}_1(U)$, we get an open and dense subset $\mathcal{B}$ of $\mathrm{RTPHA}_1(U)$ such that, for every $f\in \mathcal{B}$, there are a periodic point $p \in \Lambda_f(U)$ and a neighborhood $\mathcal{U}$ of $f$ such that $ \Lambda_g(U) \subset H(p_g, g)$ for every $g \in \mathcal{U}$.  On the other hand, the attractor contains the closure of their unstable manifolds. In particular, they contain their homoclinic classes, so $H(p_g,g) \subset \Lambda_g(U)$ for every $g \in \mathcal{U}$. Therefore, $  \Lambda_g(U) = H(p_g,g)$ for every $g \in \mathcal{U}$.

Since homoclinic classes vary lower semicontinuously and attractors vary upper semicontinuously (see Remarks~\ref{cont.isolated} and \ref{r.HC.properties}), the sets $\Lambda_f(U)$, which are both attractors and homoclinic classes, vary continuously on $f \in \mathcal{B}$.  
\end{proof}

\smallskip

\section{Invariant extensions of the partially hyperbolic splitting}\ \label{ss.extension} 

\medskip

A partially hyperbolic splitting $E^s \oplus E^c \oplus E^u$ defined over a compact invariant set $\Lambda$ can always be extended to a continuous splitting in a neighborhood of $\Lambda$. Although its not always possible to make this extension $Df$-invariant, in some particular cases it is indeed feasible, as we see in the next theorem.

 \begin{teo}[\cite{B12,B17}] \label{extension} \em{There is a residual subset $\mathcal{R}_3 \subset \operatorname{Diff}^1(M)$ with the following property. Let $f \in \mathcal{R}_3$ and $\Lambda=H(p,f)$ be a partially hyperbolic homoclinic class. Then we can extend (not uniquely) the splitting on $\Lambda$ to a partially hyperbolic splitting on a compact neighborhood $U$ of $\Lambda$ that is invariant in the following sense: for every $x \in U$ such that $f(x) \in U$, we have that $Df_x (E^i (x))= E^i(f(x)) \mbox{, for}\: \;i \in\{s,c,u\}$.} 
\end{teo}

Theorem \ref{extension} is a combination of two important results. Firt, this extension holds for a class of sets known as \emph{chain recurrent sets} (see Theorem 7 of \cite{B12}). Secondly, $C^1$-generic homoclinic classes coincides with its chain recurrent classes (see  Remark 1.10 of \cite{B17}).

\begin{re} \label{otimo} \em{In the case that the homoclinic class is an attractor, Theorem~\ref{extension} holds  openly in $\operatorname{Diff}^1(M)$. The reason is that  transitive  attractors are chain recurrent classes. Hence, in the case of attractors we do not need to invoke Remark 1.10 of \cite{B17} to apply Theorem 7 of \cite{B12}. }
\end{re}

The next two proposition do not assume generic arguments, but requires the homoclinic class to be a chain recurrence class, which holds $C^1$-genericaly. The first is a consequence of Lemma 3.6 in \cite{B12}.  

\begin{pr} \label{ttes} Let  $\Lambda = H(p,f)$ be a partially hyperbolic homoclinic class that is also a chain recurrent class with $\operatorname{index}(p)=d^s$. Let $U$ be a compact neighborhood of $\Lambda$ with an extended invariant splitting $E^s \oplus E^c \oplus E^ u$. Then, for every $x \in \Lambda$, the leaf $\mathcal{F}^s(x)$ is tangent to $E^s$ at every point in~$U$.     
\end{pr}

\begin{proof} By Lemma 3.6 of \cite{B12}, $\mathcal{F}^s(p)$ is tangent to $E^s$ at every point in $U$. Since $x \in H(p,f)$, there exist a sequence $\{x_n\}_{n \in \mathbb{N}}$ of transverse homoclinic points of $p$ that accumulates at $x$. By the continuity of the strong stable lamination in $\Lambda$, for every $r>0$, the disk sequence $\{\mathcal{F}^s_{r}(x_n)\}_{n \in \mathbb{N}}$ accumulates (with respect to the $C^1$-topology) to the disk $\mathcal{F}^s_{r}(x)$. Since $\mathcal{F}^s_{r}(x_n) \subset \mathcal{F}^s(p)$, the disks $\mathcal{F}^s_{r}(x_n)$ are tangent to $E^s$ in $U$, so the same holds for  $\mathcal{F}^s_{r}(x)$. By the arbitrary choice of $r$, we conclude the proof of the proposition. 
\end{proof}

\begin{re} \label{r.intersecta.unstable} \em{Observe that Proposition~\ref{ttes} implies that, if $\mathcal{F}^s(x)$ accumulates to a hyperbolic periodic point $q \in \Lambda$ of index $d^s$, then it  intersects the local unstable manifold of $q$. This property holds robustly for generic isolated homoclinic classes or attractors, since they are, robustly, chain recurrent classes (see Corollary 1.13 of \cite{B17}).}
\end{re}

\begin{pr} \label{p.saturation} 
Under the hypotheses of Proposition~\ref{ttes}, if  $y \in \overline{\mathcal{F}^s(x)} \cap \Lambda$, then $\mathcal{F}^s(y) \subset \overline{\mathcal{F}^s(x)}$. 
\end{pr}

\begin{proof}
Fix $n \in \mathbb{N}$ and let $z=f^n(y)$. Clearly $ \mathcal{F}^s(f^n(x))$ accumulates to $z$. Consider a sequence $\{x_k\}_{k \in \mathbb{N}}$ of points in $\mathcal{F}^s(f^n(x))$ converging\footnote{We do not know \emph{a priori} if the sequence $\{x_k\}_{k \in \mathbb{N}}$ could be taken inside $\Lambda$. That is the reason why this proposition is not an immediate consequence of the continuity of $\mathcal{F}^s$.} to the point $z$, and fix $r>0\,$ sufficiently small so that any disk of radius $r$ and centered at some point of $\Lambda$ lies inside $U$. Consider the sequence of disks $\{\mathcal{F}^s_r(x_k)\}_{k \in \mathbb{N}}$. By Proposition \ref{ttes}, each disk $\mathcal{F}^s_r(x_k)$ is tangent to the extended bundle $E^s$ in $U$. Hence, we can take a subsequence of it that converges (in the $C^1$-topology) to a $C^1$-topological disk $D$ that is also tangent to $E^s$. 

\smallskip

\begin{afir} $D = \mathcal{F}_r^s(z)$.
\end{afir}

\begin{proof}[Proof of the claim] \label{cclaim}

As $D$ is tangent to $E^s$, the set $f^m(D)$ must contract exponentially fast to $f^m(z)$. From the Hirsh-Pugh-Shub theory (see \cite{B1}, Theorem 5.4) the set $\mathcal{F}^s_r(z)$ characterises the points near $z$ with this asymptotic behavior, so $D \subset \mathcal{F}^s(z)$. On the other hand, $D$ is a topological manifold of dimension $d^s$ and radius $r$, so $D = \mathcal{F}_r^s(z)$.
\end{proof}

\smallskip

By this claim, $\mathcal{F}_r^s(z) \subset \overline{\mathcal{F}^s(f^n(x))}$. Taking the n-th pre-image on this inequality, and using the invariance of the lamination, we obtain that

$$f^{-n}\mathcal{F}_r^s(f^n(y)) \subset \overline{\mathcal{F}^s(x)}.$$
As this holds for every $n \in \mathbb{N}$, we conclude that $\mathcal{F}^s(y) \subset \overline{\mathcal{F}^s(x)}$. \end{proof}

In what follows, we fix the set $\mathcal{R} = \mathcal{R}_0 \cap \mathcal{R}_1 \cap \mathcal{R}_2 \cap \mathcal{R}_3$ and assume that the isolating block $U$ of an attractor $\Lambda_f(U)$ is always endowed with an extended partially hyperbolic splitting. 
\smallskip

\section{Generic absence of local strong homoclinic  intersections }
\bigskip

In what follows $\Lambda_f(U)$ is a partially hyperbolic set with $d^c\geq1$.  \smallskip

Given $n \in \mathbb{N}$, let $\operatorname{Per}(n , f_{|_U})$ be the set of periodic points $p \in\Lambda_f(U)$ with $\pi(p) \leq n$. By item (1) of Theorem~\ref{t.generic}, for every $f \in \mathcal{R}$ and $n \in \mathbb{N}$, the set $\operatorname{Per}(n , f_{|_U})$ is a finite hyperbolic set.

\begin{re} \label{r.per_n} \em{Given $f \in \mathcal{R}$ and $n \in \mathbb{N}$, there is a neighborhood $\mathcal{U}_n$ of $f$ such that, for every $g \in \mathcal{U}_n$, the set $\operatorname{Per}(n ,  g_{|_U})$ consists of the continuation of  $\operatorname{Per}(n,  f_{|_{U}})$ as a hyperbolic set.}
\end{re}

In the next two lemmas we use a standard Kupka-Smale-like argument to guarantee that some local strong stable and strong unstable disks do not intersect each other.

\begin{lema} \label{l.per} Let $\mathcal{U}$ be a compatible neighborhood of $f \in \mathcal{R}$ with respect to the attractor $\Lambda_f(U)$. Fixed $\varepsilon > 0$, there is a residual subset $\mathcal{G}$ of $\mathcal{U}$ such that, for every $g \in \mathcal{G}$ and every pair of distinct periodic points $a,b$ of $\Lambda_g(U)$, it holds that
$$\mathcal{F}^s_{\varepsilon}(a,g) \cap \mathcal{F}_{\varepsilon}^u(b,g) = \emptyset. $$
\end{lema}

\begin{proof}
This lemma follows from the following claim.

\begin{afir} \label{l.per_n} Fix $n \in \mathbb{N}$. Let $\mathcal{U}$ be a compatible neighborhood of $f \in \mathcal{R}$ such that, for $g \in \mathcal{U}$, the set $\operatorname{Per}(n ,  g_{|_U})$ is the continuation of the hyperbolic set $\operatorname{Per}(n ,  f_{|_U})$. Fixed $\varepsilon>0$, there is an open and dense subset $\mathcal{V}$ of $\mathcal{U}$ such that, for every $g \in \mathcal{V}$ and every pair of distinct points $p_g,q_g \in \operatorname{Per}(n ,  g_{|_U})$ it holds that
$$\mathcal{F}^s_{\varepsilon}(p_g,g) \cap \mathcal{F}_{\varepsilon}^u(q_g,g) = \emptyset. $$
\end{afir}

\begin{proof}[Proof of the claim] Fix $p,q \in \operatorname{Per}(n ,  f_{|_U})$ with $p \neq q$. By the continuous dependence of the leafs on $g$, if $\mathcal{F}^s_{\varepsilon}(p_g,g) \cap \mathcal{F}_{\varepsilon}^u(q_g,g) = \emptyset$, then the same holds for any diffeomorphism in an open neighborhood of $g$. 

\smallskip

On the other hand, since $d^u+d^s$ is less than the ambient dimension, if $\mathcal{F}^s_{\varepsilon}(p_g,g) \cap \mathcal{F}_{\varepsilon}^u(q_g,g) \neq \emptyset$, then this intersection is not transverse. Hence, after an arbitrarily small perturbation, we can assume that the disks $\mathcal{F}^s_{\varepsilon}(p_g,g)$ and $\mathcal{F}_{\varepsilon}^u(q_g,g)$ are disjoint. As a conclusion, there is an open and dense subset $\mathcal{V}_{p,q}$ of $\mathcal{U}$ such that
\begin{equation} \label{e.su.empty} \mathcal{F}^s_{\varepsilon}(p_g,g) \cap \mathcal{F}_{\varepsilon}^u(q_g,g) = \emptyset,  \mbox{ for every } g \in \mathcal{V}_{p,q}.  \end{equation}

By Remark~\ref{r.per_n}, the set $$ \displaystyle \mathcal{V} = \bigcap_{(p,q) \in \mathcal{B}} \mathcal{V}_{p,q} \ ,  \mbox{ where } \mathcal{B} = \{(p,q) \in \{\operatorname{Per}(n ,  f_{|_U})\}^2  \ | \ p \neq q \},$$ is a finite intersection of open and dense subsets of $\mathcal{U}$, which means that $\mathcal{V}$ is also  open and dense in $\mathcal{U}$. By construction, the set $\mathcal{V}$ satisfies the required property in this Claim.~\end{proof}

To conclude the proof of the Lemma~\ref{l.per}, we use a genericity argument as follows. Let $\{f_i\}_{i\in \mathbb{N}}$ be a dense subset of $\mathcal{U}\cap \mathcal{R}$. Fixed $n \in \mathbb{N}$, we apply Claim~\ref{l.per_n} to each $f_i$, $i \in \mathbb{N}$. In this way, we obtain an open and dense subset $\mathcal{V}_i^n$ of a neighborhood of $f_i$, satisfying the non intersection condition in Claim~\ref{l.per_n}. Note that $\mathcal{G}_n = \bigcup_{i \in \mathbb{N}} \mathcal{V}_i^n$ is an open and dense subset of $\mathcal{U}$. Finally, we set $\mathcal{G} = \bigcap_{n \in \mathbb{N}} \mathcal{G}_n$, which is the desired residual subset of $\mathcal{U}$ in Lemma~\ref{l.per}. \end{proof}

\section{Minimality}
\label{minimality}

For notational simplicity, given a strongly partially hyperbolic set $\Lambda$  we adopt the following notation. $$\mathcal{F}^{s}_{\Lambda}(x) = \mathcal{F}^{s}(x) \cap \Lambda  \quad \mbox{and} \quad \mathcal{F}^{u}_{\Lambda}(x) = \mathcal{F}^{u}(x) \cap \Lambda. $$

\begin{de}[$s$ and $u$-minimal laminations] \label{def.minimal} \em{Let $\Lambda$ be a  partially hyperbolic set of a diffeomorphism $f$ with nontrivial stable bundle $E^s$. We say that the lamination $\mathcal{F}^{s}$ is \emph{minimal} if there is $d \in \mathbb{N}$ such that,
for all $x \in \Lambda$, it holds that
$$
\bigcup_{i=1}^{d} \overline{\mathcal{F}_{\Lambda}^{s}(f^i(x))} = \Lambda.
$$ 
In this case we say that $\Lambda$ is an {\emph{$s$-minimal set}}.

We say that  $\Lambda$ is a \emph{robustly $s$-minimal set} if $\Lambda=\Lambda_f(U)$ is an isolated set, and  $\Lambda_g(U)$ is $s$-minimal for all $g$ in a neighborhood $\mathcal{U}$ of $f$. If $s$-minimality is verified only in a residual subset of $\mathcal{U}$, then we say that $\Lambda_f(U)$ is a \emph{generically $s$-minimal set}.

The definition of $u$-minimality is analogous, considering the 
strong unstable lamination $\mathcal{F}^{u}$. }
\end{de}

\smallskip


\subsection{A criterion for minimality}\
\medskip

Here we stablish a criterion to verify $u$- or $s$-minimality on homoclinic classes.  

\begin{teo}[criterion for minimality] \label{t.criterion2}  Let $f \in \text{Diff}^1(M)$ and $\Lambda = H(p,f)$ be a  partially hyperbolic homoclinic that is also a chain recurrent class with $\operatorname{index}(p)= d^s$ . Let $U$ be a compact neighborhood of $\Lambda$ with an extended invariant splitting $E^s \oplus E^c \oplus E^u$. If   $p \in \overline{\mathcal{O}_f(\mathcal{F}^{s}(x))}$ for every $x\in \Lambda$, then $\Lambda$ is $s$-minimal.
\end{teo}

\begin{cor}[generic version] \label{t.criterion}  Let $f \in \mathcal{R}$ and $\Lambda = H(p,f)$ be a partially hyperbolic homoclinic class with $\operatorname{index}(p)= d^s$. If $\overline{\mathcal{O}_f(\mathcal{F}^{s}(x))} \cap \operatorname{Per}_{d^s}(f_{|_{\Lambda}}) \not=~\emptyset$ for every $x\in \Lambda$, then $\Lambda$ is $s$-minimal.
\end{cor}

\begin{re} \em{\label{R0} 
There are dual statements considering the  
unstable lamination of $H(p,f)$ when $\operatorname{index}(p) =d^s+d^c$.} 
\end{re}

To prove these criterions we need an auxiliary lemma. Here we only treat the $s$-minimal case.

\begin{lema} \label{l.restriction} Let $\Lambda$ be as in Theorem~\ref{t.criterion2} and $\pi(p)=d$. If there is $x \in \Lambda$ such that $p \in \overline{\mathcal{F}^s(x)}$, then $
\Lambda =  \bigcup_{i=1}^{d} \overline{\mathcal{F}^s_{\Lambda}(f^i(x)) }.$
\end{lema}

\begin{proof}

Note that the inclusion $\bigcup_{i=1}^{d} \overline{\mathcal{F}^s_{\Lambda}(f^i(x))} \subset \Lambda$ is immediate (recall the notation $\mathcal{F}^s_{\Lambda}(x) = \mathcal{F}^ s(x) \cap \Lambda$). To prove that $\Lambda \subset \bigcup_{i=1}^{d} \overline{\mathcal{F}^s_{\Lambda}(f^i(x)) }$, we first observe that $\Lambda \subset  \bigcup_{i=1}^{d} \overline{\mathcal{F}^s_{\Lambda}(f^i(p))}$.    To see this last inclusion, note that as $\pi(p)=d$ and $\operatorname{index}(p)=d^s$, we have 
\begin{equation} \label{ee}
\Lambda = H(p,f)= \overline{W^s(\mathcal{O}_f(p)) \cap \Lambda} = \bigcup_{i=1}^{d} \overline{\mathcal{F}^s_{\Lambda}(f^i(p))}.
\end{equation}

As $p \in \overline{\mathcal{F}^s(x)}$, Proposition~\ref{p.saturation} gives that  $\mathcal{F}^s(p) \subset \overline{\mathcal{F}^s(x)}$. 

\smallskip

We claim that $\mathcal{F}^s_{\Lambda}(p) \subset \overline{\mathcal{F}_{\Lambda}^s(x)}$. Indeed, given $z \in \mathcal{F}^s_{\Lambda}(p)$, consider a transverse homoclinic point $\tilde{z}$ of $p$ close to $z$. Since $\mathcal{F}^s(p) \subset \overline{\mathcal{F}^s(x)}$, the leaf $\mathcal{F}^s(x)$
accumulates at $\tilde{z}$ and intersect $W^u(\mathcal{O}_f(p))$ at a point $w$ that can be chosen arbitrarily close to $\tilde{z}$. By Equation~\eqref{ee}, there is $j\in \mathbb{N}$ such that $\mathcal{F}^s(f^j(p))$ accumulates at $x$ and, by Proposition~\ref{p.saturation}, it also accumulates at $w$. Then, $\mathcal{F}^s(f^j(p))$ meets transversely $W^u(\mathcal{O}_f(p))$ in a sequence of homoclinic points of $p$ converging to $w$,  which means that $w \in H(p,f)$. By consctruction, $w \in \mathcal{F}_{\Lambda}^s(x)$ can be taken arbitrarily close to $z$, so $z \in \overline{\mathcal{F}_{\Lambda}^s(x)}$. From the arbitrary choice of $z \in \mathcal{F}^s_{\Lambda}(p)$, we conclude that $\mathcal{F}^s_{\Lambda}(p) \subset \overline{\mathcal{F}_{\Lambda}^s(x)}$.

\smallskip

This last inclusion and Equation~\eqref{ee} imply that $\Lambda \subset \bigcup_{i=1}^{d} \overline{\mathcal{F}^s_{\Lambda}(f^i(x))}$, finishing the proof of the lemma. \end{proof}  

In what follows, given a compact set $X \subset M$, we denote by $B_{\varepsilon}(X)$ the set of  points in $M$ whose distance to $X$, with respect to a Riemannian metric on $M$, is less than $\varepsilon$.

\begin{proof}[Proof of Theorem~\ref{t.criterion2}] Fix $x \in \Lambda$. Since $p \in \overline{\mathcal{O}(\mathcal{F}^{s}(x))}$, given $\varepsilon >0$ there is $j_1 \in \mathbb{Z}$ such that $p \in B_\varepsilon(\mathcal{F}^{s}(f^{j_1}(x))).$ 
By Proposition \ref{ttes}, the leaf $\mathcal{F}^{s}(f^{j_1}(x))$ intersects transversely the local unstable manifold of $p$, provided $\varepsilon$ is small enough. Hence, by the $\lambda$-lemma, there is
 $j_2(x) \in \mathbb{N}$ such that, for every $j \geq j_2(x)$, it holds that
$$
\Lambda \subset B_{\frac{\varepsilon}{2}} \Big(\bigcup_{i=1}^d 
\mathcal{F}^{s}(f^{-j + i}(x))\Big),\quad \mbox{where\ } \pi(p)=d.
$$
By the continuity of the lamination $\mathcal{F}^{s}$ and the $\lambda$-lemma,
there is a neighborhood $U_x$ of $x$ satisfying
\begin{equation} \label{e.Be}
\Lambda\subset 
B_\varepsilon \Big(
\bigcup_{i=1}^{d} \mathcal{F}^{s}(f^{-j + i}(y))
\Big), \quad \mbox{for all $y \in U_x\cap \Lambda $ and $j \geq j_2(x)$}.
\end{equation}

In this way, for each $x\in \Lambda$
we get a number $j_2(x)$ and a neighborhood $U_x$ of $x$ satisfying Inclusion~\eqref{e.Be}. Using these open sets as a covering for the compact set $\Lambda$, we extract a finite subcovering  $\bigcup_{i=1}^m U_{x_i}$ of $\Lambda$.

\smallskip

Set $J = \max_{i=1}^m \{j_2(x_i)\}$. 
By construction, 
\begin{equation} \label{e.Bey}
\Lambda \subset 
B_\varepsilon \Big( \bigcup_{i=1}^d \mathcal{F}^{s}(f^{-J + i}(y))\Big), 
\quad
\mbox{for all $y \in \Lambda $}.
\end{equation}

Applying  Inclusion~\eqref{e.Bey} to $y =f^{J}(x)$, and observing that $\varepsilon$ can be taken arbitrarily small, 
we conclude that  $\Lambda \subset\bigcup_{i=1}^{d} \overline{\mathcal{F}^{s}(f^i(x))}$. Then, there is $j_3(x) \in \{1,...,d\}$ such that
$p \in \overline{\mathcal{F}^{s}(f^{j_3(x)}(x))} \mbox{,  which means that  } f^{-j_3(x)}(p) \in \overline{\mathcal{F}^{s}(x)}$. Applying Lemma~\ref{l.restriction} to the periodic point $f^{-j_3(x)}(p)$, and observing that $H(p,f) = H(f^{-j_3(x)}(p), f) = \Lambda$, we obtain that
$\bigcup_{i=1}^d \overline{\mathcal{F}^{s}_{\Lambda}(f^i(x))} = \Lambda.
$ 
As it holds for all $x \in \Lambda$, the set $\Lambda$ is $s$-minimal.
\end{proof}

\begin{proof}[Proof of Corollary~\ref{t.criterion}] 

Since we are in the gereric context, we can assume that $\Lambda$ is a chain recurrent class admitting an invariant extension of the splitting to a compact neighborhood $U$ (see Theorem~ \ref{extension}).  By hypothesis, for every $x \in \Lambda$ there is a point  $p_x \in \overline{\mathcal{O}_f(\mathcal{F}^{s}(x))} \cap \operatorname{Per}_{d^s}(f_{|_{\Lambda}})$. From the invariance of the set $\overline{\mathcal{O}_f(\mathcal{F}^{s}(x))}$ and the $\lambda$-Lemma, we get that
\begin{equation} \label{e.px1}
 \overline{\mathcal{O}_f(\mathcal{F}^s(p_x))} \subset \overline{\mathcal{O}_f(\mathcal{F}^{s}(x))}.
 \end{equation}
Note that $\mathcal{F}^s(p_x) = W^s(p_x)$, so   
\begin{equation} \label{e.px2} 
H(p_x,f) \subset \overline{\mathcal{O}_f(\mathcal{F}^s(p_x))}.
\end{equation}

By item (3) of Theorem~\ref{t.generic}, every two non-disjoints homoclinic classes coincide, so $\Lambda = H(p_x,f)$. Putting together this fact, equations \eqref{e.px1} and \eqref{e.px2}, we obtain 
that $\Lambda \subset \overline{\mathcal{O}_f(\mathcal{F}^{s}(x))}.$

As this holds for every $x \in \Lambda$, we apply Theorem~\ref{t.criterion2} and conclude that $\Lambda$ is $s$-minimal. 
\end{proof}

\medskip

\section{Generic $u$- and $s$-minmal attractors}
 
\medskip

A first step in proving Theorems A and B is to classify the dynamics on certain central invariant curves. According to this classification, we investigate $u$ or $s$-minimality case by case. 

\subsection{Central Curves: Classification of Periodic Points}\ \label{ccpp}

\medskip

Unlike the strong stable and unstable bundles, we can not guarantee the existence of an invariant center lamination 
tangent to the center bundle. Nevertheless, if $\Lambda$ is a partially hyperbolic attractor with one-dimensional center bundle, we guarantee the existence of \emph{invariant central curves} for the hyperbolic periodic points of $\Lambda$ (Proposition \ref{HHU}). A {\emph{central curve}} is a curve $\gamma \subset M$ that is tangent to the (extended) center bundle $E^c$ at every point of $\gamma \cap U$ (see subsection~\ref{ss.extension}). 
\smallskip

Next result is an adaptation\footnote{In \cite{B10}, the partial hyperbolicity is defined over the whole manifold.} of Theorem 2 in \cite{B10} for the context of partially hyperbolic attractors.

\begin{pr}
\label{HHU} Let $f \in \mathcal{R}$ and $\Lambda_f(U)$ be a partially hyperbolic attractor of $f$ with one-dimensional center bundle. Then there exists $K>0$ such that, for every hyperbolic periodic point with period $N \geq K$, there exists an $f^N$-invariant central curve $L(p)$ (i.e., $f^N(L(p))=L(p)$) containing $p$ in its interior.  
\end{pr}

\begin{proof}[Sketch of the proof] The proof of this result is almost identical to the one in Theorem 2 of \cite{B10}. It involves only local arguments, which are still valid inside the isolating block $U$ of the attractor $\Lambda_f(U)$. Following their arguments, for each periodic point $p \in \Lambda_f(U)$ with period $N$ sufficiently big, we obtain a local central curve $\gamma(p)$ inside $U$ with the following property: $\gamma(p)\backslash{p}$ is the union of two connected components $\gamma^+(p)$ and $\gamma^-(p)$, both invariant either by $f^N$ or $f^{-N}$. By taking forward and backward iterates of these components, we can extend $\gamma(p)$ to a curve $L(p)$ such that $f^N(L(p))=L(p)$.  In the process, we may assume that $L(p)$ do not end in a periodic point inside $U$ by extending $L(p)$ if necessary. 
\end{proof}


\begin{re} \label{r.finite} \em{ Recall that $\mathcal{R}$ consists of Kupka-Smale diffeomorphisms, so the set of periodic points $p$  with $\pi(p) \leq K$ is a finite set. } 
\end{re}

Note that if $\pi(p)=d$, then the period of any periodic point in the curve  $L(p)$ is a divisor of $2d$. Hence, there are only finitely many periodic points in $L(p)$.

In general, there is not a unique invariant central curve passing through $p$. We consider a choice of these invariant curves that is coherent with the dynamics on $\Lambda_f(U)$, that is, satisfying $ L(f(p))=f(L(p)) $.

\smallskip

 We denote  by $L_{U}(p)$ the connected component 
of $L(p) \cap U$ containing $p$ and by $\Gamma_p \subset L_U(p)$  the smallest compact and connected subset of $L_U(p)$ that contains all periodic points 
and all periodic closed curves of $L_U(p)$ (it may happens that $\Gamma_p = \{p\}$). Given $\varepsilon>0$, we denote by $L_{\varepsilon }(p)$ the connected component of $L_U(p) \cap B_{\varepsilon}(p)$ that contains $p$.  There are three possibilities 
for the boundary $\partial \Gamma_p$ of $\Gamma_p$ relative to the set $L_U(p)$: 
either it is empty,  a unitary set, or a two points set. If $\partial \Gamma_p = \emptyset$, then $\Gamma_p$ is a closed curve. When $\partial \Gamma_p \not= \emptyset$, we say that $\partial \Gamma_p$ are the \emph{extremal points} of $\Gamma_p$. 
A periodic point $q$ is called \emph{extremal} if there is some $p \in \Lambda_f(U)$ such that $q\in \partial \Gamma_p$.

\begin{re}\label{r.length} \em{Since $U$ is a neighborhood of the compact set $\Lambda_f(U)$, the length of $L_U(p)$ is uniformly bounded from below, and the point $p$ is uniformly far from the edges of $L_U(p)$, if any. Hence, there is $\delta > 0$ such that, for every periodic point $p \in \Lambda$, the set $L_U(p)$ contains a disk centered at $p$ of length bigger than $\delta$.}
\end{re}

Now we classify the periodic points of $f$ in $U$ as follows:
$$
 P_1 \cup P_2 \cup P_3 \cup P_4 =\{p\in \operatorname{Per}(f)\cap U, \quad 
\pi(p) \geq K \}
$$ 
where
 \begin{itemize}
	\item $p\in P_1$ if the extremal points of $\Gamma_p$ are attracting  in the central direction,  
	\item $p\in P_2$ if the extremal points of $\Gamma_p$ are repelling in the central direction,  
	\item $p\in P_3$ if there are one attracting and one repelling extremal points of $\Gamma_p$, and
	\item $p\in P_4$ if $\Gamma_p$ is a closed curve.
   
\end{itemize} 

\begin{re}\label{r.union} \em{

Since the dynamics in $L(p)$ is Morse-Smale, there are finitely many
periodic points $a_1, \dots,a_{m_p}$ such that
$$
\Gamma_p  \subset \bigcup_{i=1}^{m_p} W^s(a_{i},f) \quad  \mbox{and}  \quad \Gamma_p  \subset \bigcup_{i=1}^{m_p} W^u(a_{i},f), $$}
\end{re}
\noindent where the stable and unstable manifolds $W^s$ and $W^u$ refer to the dynamics restricted to $L(p)$. Observe that  $W^s(a_i,f) = \{a_i\}$ if $\operatorname{index}(a_i) = d^s$, and  $W^u(a_i,f) = \{a_i\}$ if $\operatorname{index}(a_i)~=~d^s+~1$.

\begin{re}\label{cl.P2P4} \em{
If $p \in P_1 \cup P_4$, then $\partial \Gamma_p$ is either the empty set or consists of (at most two) points of index $d^s+1$. Hence
$$
L_U(p) \subset \Gamma_p \cup W^s (\partial \Gamma_p).
$$

Similarly, if $p \in P_2 \cup P_4$, then $\partial \Gamma_p$ is either the empty set or consists of (at most two) points of index $d^s$. Hence
$$
L_U(p) \subset \Gamma_p \cup W^u (\partial \Gamma_p) \subset \Lambda.
$$}
\end{re}

\begin{lema} \label{lrg} 
 For every $p \in \operatorname{Per}(f) \cap U$ with $\pi(p) >K$, the following holds:
\begin{enumerate}
\item $\Gamma_p \subset \Lambda_f(U)$, 
\item $f(\Gamma_p) = \Gamma_{f(p)}$,
\item $f(P_i) = P_i$.
\end{enumerate}
\end{lema}

\begin{proof}
By definition, the periodic points of
$\Gamma_p$ belong to $U$. By Remmark~\ref{r.union}  
$$
\Gamma_p \subset \bigcup_{q\in \Gamma_p} W^u(q) \subset \Lambda_f(U).
$$
The last inclusion follows from the fact that $\Lambda_f(U)$ is an attractor, so it contain its unstable sets.  This proves item (1).

\smallskip

Items (2) and (3) follow from the coherent choice of the central curves. Observe that $f$ sends 
closed curves to closed curves 
and extremal points of $\Gamma_p$ to extremal 
points of $\Gamma_{f(p)}$. 
\end{proof}

\begin{lema} \label{l.87} $\overline{P_i} = \Lambda_f(U)$ for some $i \in \{1,...,4\}$.
\end{lema}

\begin{proof}
By item (2) of Theorem~\ref{t.generic}, the periodic points of $f$ are dense in $\Lambda_f(U)$. By Remark~\ref{r.finite}, the periodic points of $\Lambda_f(U)$ with period less than $K$ is a finite set. Since $\Lambda_f(U)$ is infinite and transitive, it has no isolated periodic orbits. Hence the set of periodic points of $\Lambda_f(U)$ with period bigger than $K$ is also dense in $\Lambda_f(U)$. In other words, $\Lambda_f(U) = \overline{P_1} \cup \overline{P_2} \cup \overline{P_3} \cup \overline{P_4}$. 

Let $x \in \Lambda_f(U)$ be a point with a dense orbit. Clearly,  $x \in \overline{P_i}$ for some $i \in \{1,...,4\}$. From item (3) in Lemma~\ref{lrg},
the whole orbit of $x$ lies in $\overline{P_i}$, which implies that $\overline{P_i} = \Lambda_f(U)$. 
\end{proof}

\bigskip

\subsection{First step toward Theorem A}\
\label{proofTA}
\medskip

Recall the notation $\mathrm{GTPHA}_1(U)$ as the subset of $\operatorname{Diff}^1(M)$ of diffeomorphisms $f$ such that $\Lambda_f(U)$ is a robustly non-hyperbolic, partially hyperbolic set with one-dimensional center bundle that is generically transitive. Given $f_0 \in \mathcal{R}\cap \mathrm{GTPHA}_1(U)$, let $\mathcal{U}_0$ be a compatible neighborhood of $f_0$, and $\mathcal{G}_0$ be the residual subset of $\mathcal{U}_0$ given by Lemma~\ref{l.per}.

Fixed $f \in \mathcal{G}_0 \cap \mathcal{R}$ and
 $\Lambda=\Lambda_f(U)$, we verify that $\Lambda$ has one minimal lamination.
We split these verification into three cases, 
according to which sets $P_i$'s  are dense
in $\Lambda$ (see Lemma~\ref{l.87}). 

\begin{pr}\label{p.cases} $\,$
\begin{enumerate}[(a)]
 \item 
If the set $P_1 \cup P_4$ is dense in $\Lambda$,
then $\Lambda$ is
 $u$-minimal.
\item 
If the set $P_2 \cup P_4$ is dense in $\Lambda$,
then $\Lambda$ is  $s$-minimal.
\item 
If the set $P_3 \cup P_4$ is dense in $\Lambda$,
then $\Lambda$ is  $u$-minimal.
\end{enumerate}
\end{pr}

Fix $x \in \Lambda_f(U)$. Since $\Lambda_f(U)$ is an attractor, for every $\varepsilon > 0$, the disk $\mathcal{F}^u_{\varepsilon}(x)$ is a subset of $\Lambda$. Hence every poin $z \in \mathcal{F}^u_{\varepsilon}(x)$ has a local stable disk $\mathcal{F}^s_{\varepsilon}(z)$, and we define the topological disk of codimension one
\begin{equation}\label{e.Delta}
\Delta(x, \varepsilon)= \displaystyle 
\bigcup_{z \in \mathcal{F}_{\varepsilon}^{u}(x)} \mathcal{F}_{\varepsilon}^{s}(z).
\end{equation}

\begin{re} \label{r.rr} \em{By Remark \ref{r.length}, for any periodic point $p$, the curve $L_U(p)$  contains a disk of length $\delta$ centered at $p$ ($\delta$ does not depend on $p$). Thus,  if $p$ is sufficiently close to $x$, then  $L_U(p)$ meets topologically transversely the disk $\Delta(x, \varepsilon)$ at some point $z_p =z_p(x)$~(see Figure \ref{ca})}. 
\end{re}

\begin{lema} \label{l.a} If $z_p$ lies in the stable manifold of some periodic point $\tilde{p} \in \Gamma_p$, then $\tilde{p} \in \overline{\mathcal{O}_f(\mathcal{F}^u(x))}$. 
\end{lema}

\begin{proof}

By the definition of $\Delta(x, \varepsilon)$ and $z_p$, there is a point $w_p \in \mathcal{F}^{u}_{\varepsilon}(x)$ such that 
$z_p \in \mathcal{F}^{s}_{\varepsilon}(w_p)$. Then, considering a Rimannian metric $d(\cdot,\cdot)$ in $U$,  
\begin{equation} \label{e.zpwp}
\lim_{n \to \infty} d(f^{n}(z_p),f^{n}(w_p)) =0.
\end{equation} 

By hypothesis, $z_p \in W^s(\tilde{p}, f)$, so we get that
\begin{equation} \label{e.zpp}
\lim_{n \to \infty} d(f^{n}(z_p),f^{n}(\tilde{p})) =0.
\end{equation} 
Combining \eqref{e.zpwp}  and \eqref{e.zpp}, we obtain  
$$\lim_{n \to \infty} d(f^{n}(\tilde{p}),f^{n}(w_p)) \to 0.$$
As $w_p \in \mathcal{F}^u(x)$, this implies that the orbit of $\mathcal{F}^{u}(x)$ 
accumulates at $\tilde{p}$. \end{proof}

 Now we start to prove Proposition~\ref{p.cases}.
\begin{proof}[Proof of (a)]

\smallskip

As $f \in \mathcal{R}$, $\Lambda$ is a homoclinic class (Remark~\ref{5gta})
. By Corollary~\ref{t.criterion} and Remark~\ref{R0}, to prove the $u$-minimality of $\Lambda$, 
 it suffices to see that, for every $x \in \Lambda$, it holds that
\begin{equation}
 \label{e.case1}
\overline{\mathcal{O}(\mathcal{F}^{u}(x))} \cap 
\operatorname{Per}_{d^s+1}(f_{|_{\Lambda}}) \not= \emptyset.   
\end{equation}

Fix $x \in \Lambda$. Since $P_1 \cup P_4$ is dense in $\Lambda$, we can take
 $p\in P_1\cup P_4$ sufficiently close to $x$ so that, as in Remark~\ref{r.rr}, there is a transverse intersection point $z_p$ between $L_U(p)$ and $\Delta(x, \varepsilon)$.

By Remarks \ref{r.union} and \ref{cl.P2P4}, the point $z_p$ either lies in a stable manifold of some periodic point $\tilde{p} \in \Gamma_p$ of index $d^s+1$, or $z_p$ is a hyperbolic periodic point of index $d^s$.

\smallskip

If the first possibility holds, then Lemma~\ref{l.a} implies Equation~\eqref{e.case1} and we are done.
\begin{figure}[hnt!] \centering
\includegraphics[scale=0.8]{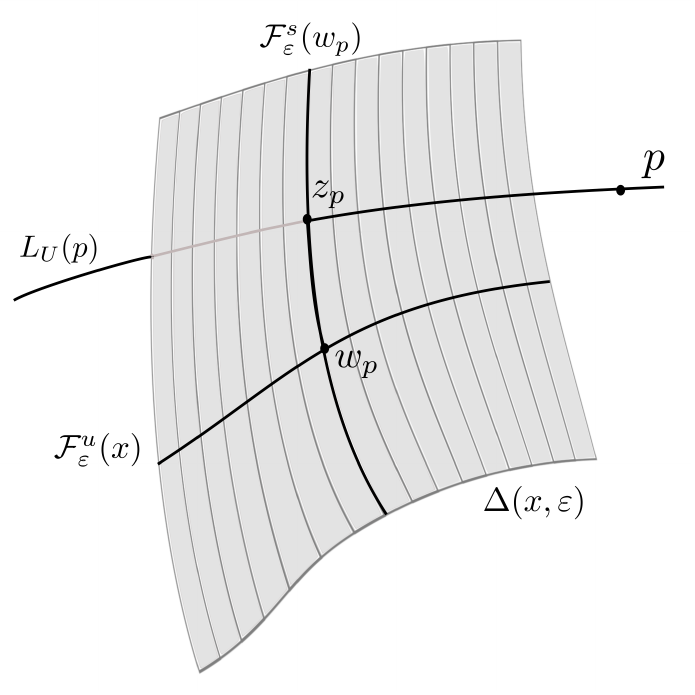}
\caption{Case (a)}
\label{ca}
\end{figure}

Let us suppose then that $z_p$ is a hyperbolic periodic point of index $d^s$. From Equation~\eqref{e.zpwp}, the orbit of $\mathcal{F}^u(x)$ accumulates at the orbit of $\mathcal{F}^u(z_p)$.  Thus, to get Equation~\eqref{e.case1} it is enough to show that 
\begin{equation} \label{e.deltazp}
\overline{\mathcal{O}(\mathcal{F}^{u}(z_p))} \cap 
\operatorname{Per}_{d^s+1}(f_{|_{\Lambda}}) \not= \emptyset.
\end{equation}

\smallskip

Consider the topological manifold $\Delta(z_p, \varepsilon)$. For every $q \in P_1 \cup P_4$ sufficiently close to $z_p$, the curve $L_U(q)$ meets topologically transversely $\Delta(z_p, \varepsilon)$ at some point $z_q$. By Lemma~\ref{l.per}, $z_p$ is the only periodic point in $\Delta(z_p, \varepsilon)$, and by Remark~\ref{r.union}, $z_q$ belongs to the stable manifold of some periodic point in $L_U(q)$ with index $d^s+1$. Now Lemma~ \ref{l.a} implies Equation~\eqref{e.deltazp} and, consequently, Equation~\eqref{e.case1}.  This end the proof of Case (a).

\bigskip

\noindent
\textit{Proof of (b).} Note that this case is not as symmetrical to the Case (a) as it seems. We cannot saturate a strong stable disk with strong unstable leaves, since not all of the points in the strong stable disk belong to $\Lambda$. Instead,  for each point in $p \in P_2 \cup P_4$, we introduce the  topological disk 
\begin{equation} \label{e.nabla}
 \nabla(p,\varepsilon) = \bigcup_{y \in L_{\varepsilon , U}(p)} \mathcal{F}^{u}_\varepsilon(y) \quad \rm{(see Figure~\ref{cb})}.
\end{equation}

Recall that in this case the curve
$L_{\varepsilon , U}(p)$ is contained in $\Lambda$ (see Remark~\ref{cl.P2P4}). Thus, for every $y \in L_{\varepsilon, U}(p)$ there exists $\mathcal{F}^u(y)$, so $\nabla(p,\varepsilon)$ is well defined. 

\smallskip

By Corollary~\ref{t.criterion}, to prove the $s$-minimality of $\Lambda$ it is enough to see that, for every $x \in \Lambda$, it holds that
\begin{equation}
 \label{e.case2}
\overline{\mathcal{O}(\mathcal{F}^{s}(x))} \cap 
\operatorname{Per}_{d^s}(f_{|_{\Lambda}}) \not= \emptyset.
\end{equation}

The curve $L_U(p)$ contains a disk centered at $p$ with length $\delta$ (Remark~\ref{r.length}). Hence, given any $x \in \Lambda$, if $p \in P_2 \cap P_4$ 
is sufficiently close to $x$, then $\mathcal{F}^{s}_\varepsilon(x)$ 
intersects (topologically transversely) $\nabla (p, \varepsilon) $ at some point $w_p$ (see Figure \ref{cb}).  
By the definition of $\nabla (p,\varepsilon)$, there is $z_p\in L_{\varepsilon, U}(p)$ such that $w_p\in \mathcal{F}^{u}_\varepsilon(z_p)$, which means that
\begin{equation} \label{e.zpwp2}
\lim_{n \to \infty} d(f^{-n}(z_p),f^{-n}(w_p)) =0.
\end{equation} 

By Remarks \ref{r.union} and \ref{cl.P2P4}, either
$z_p\in W^u(\tilde p)$ for some periodic point  $\tilde{p} \in \Gamma_p \cap \operatorname{Per}_{d^s}(f_{|_{\Lambda}})$, or $z_p \in \operatorname{Per}_{d^s+1}(f_{|_{\Lambda}})$.
In the first situation, we get that 
\begin{equation} \label{e.zpp2}
\lim_{n \to \infty} d(f^{-n}(z_p),f^{-n}(\tilde{p})) =0 .
\end{equation}
Combining Equations \eqref{e.zpwp2}  and \eqref{e.zpp2} it gives that $$\lim_{n \to \infty} d(f^{-n}(\tilde{p}),f^{-n}(w_p)) \to 0 ,$$
which implies Equation~\eqref{e.case2}, as $w_p \in \mathcal{F}^s(x)$.

\smallskip 

 \begin{figure}[hnt!] \centering
\includegraphics[scale=0.35]{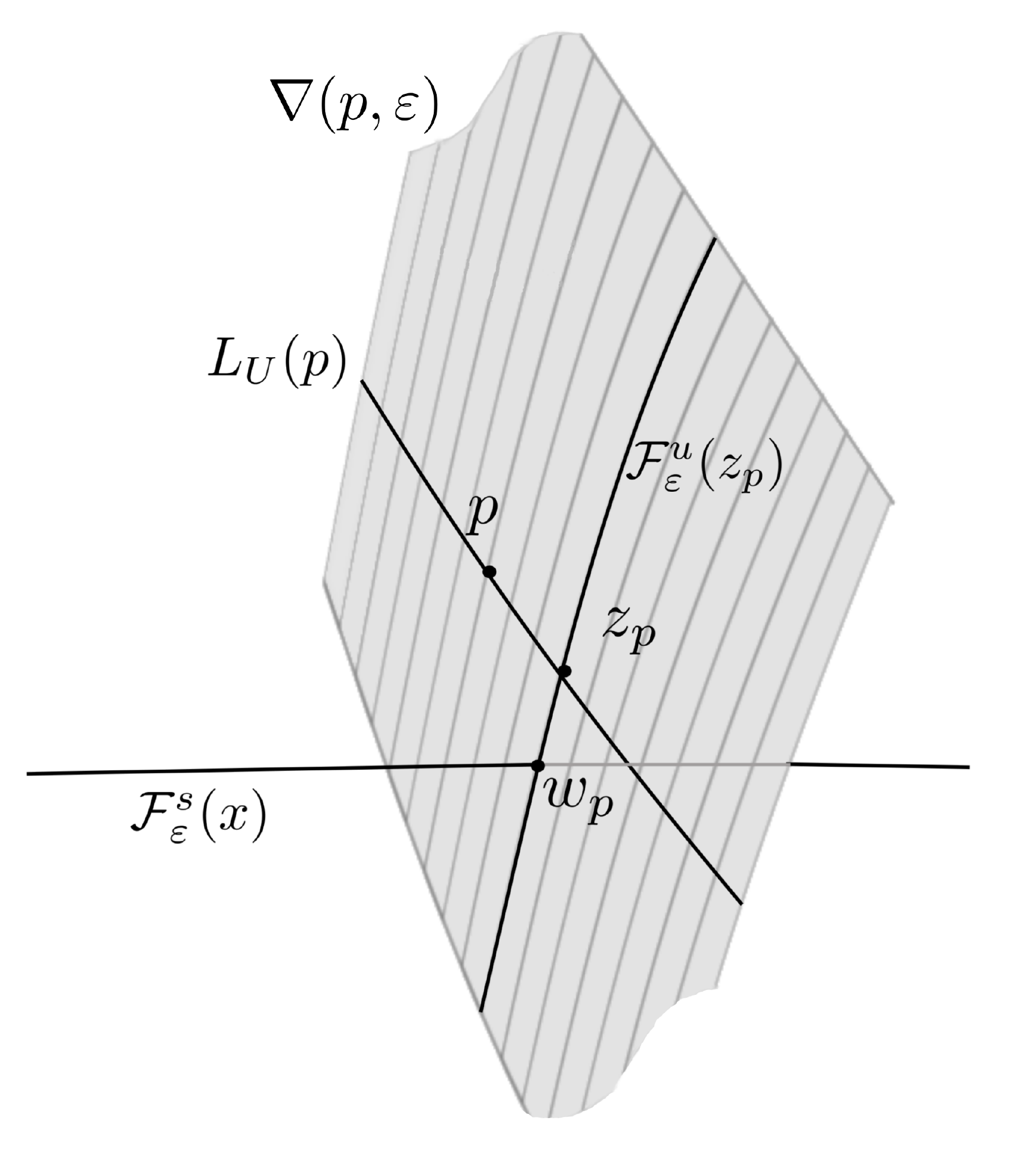}
\caption{Case (b)}
\label{cb}
\end{figure}  

 If $z_p \in \operatorname{Per}_{d^s+1}(f_{|_{\Lambda}})$, then the orbit of $\mathcal{F}^s(x)$ accumulates at the orbit of $\mathcal{F}^s(z_p)$. Thus, to get Equation~\eqref{e.case2}, it is enough to prove that
\begin{equation} \label{e.nablazp}
\overline{\mathcal{O}(\mathcal{F}^{s}(z_p))} \cap 
\operatorname{Per}_{d^s}(f_{|_{\Lambda}}) \not= \emptyset.
\end{equation}

Consider a periodic point $q \in P_2 \cup P_4$ close to $z_p$ such that $\mathcal{F}^s_{\varepsilon}(z_p)$ intersects topologically transversely $\nabla(q, \varepsilon)$ at a point $w_q$. Then, there is a point $z_q \in L_U(q)$ with $w_q \in \mathcal{F}^u_{\varepsilon}(z_q)$. By Lemma~\ref{l.per}, $z_q$ is not periodic, so $z_q$ lies in the unstable manifold of some periodic point of index $d^s$ in $\Gamma_q$. In this situation, Equations \eqref{e.zpwp2}  and \eqref{e.zpp2} hold replacing $z_p$, $w_p$, and $\tilde{p}\,$ by $z_q$,$w_q$, and $q$, respectively. Arguing as before, these equations leads to Equation~\eqref{e.nablazp} and, consequently, Equation~\eqref{e.case2}. This end the proof of Case (b).

\bigskip

\noindent
\textit{Proof of (c).} By Corollary~\ref{t.criterion}, to prove the $u$-minimality of $\Lambda$ 
it is enough to see that, for every $x\in \Lambda$, Equation~\eqref{e.case1} holds.

\smallskip

Consider the codimension one topological disk  
$\Delta(x, \varepsilon)$ as in Equation \eqref{e.Delta}. Fix $\tilde{\varepsilon}>0$ and $p \in P_3$ sufficiently close to $x$ so that $L_{\tilde{\varepsilon} ,U}(p)$ intersects topologically transversely $\Delta(x, \varepsilon)$ at a point $z_p$.

Let $l$ be the curve joining $z_p$ and $p$ inside $L_U(p)$. We can assume that there is no periodic points in the interior of $l$ (otherwise, we replace $p$ by a periodic point in $L_U(p)$ with this property).

\smallskip

There are three possible situations: 

\smallskip

\begin{enumerate}[(i)]
\item $\triangle (x, \varepsilon) \cap \operatorname{Per}_{d^s}(f_{|_{\Lambda}}) = \emptyset\ $ and $\ z_p \in W^s(p)\,$.
\smallskip

\item  $\triangle (x, \varepsilon) \cap \operatorname{Per}_{d^s}(f_{|_{\Lambda}}) = \emptyset\ $ and $\ z_p \in W^u(p)\,$.

\item  $\triangle (x, \varepsilon) \cap \operatorname{Per}_{d^s}(f_{|_{\Lambda}}) \not= \emptyset\ $. 
\end{enumerate}
\smallskip

In (i), we have $\operatorname{index}(p) = d^s +1$, so we can apply Lemma~\ref{l.a} to obtain   Equation~\eqref{e.case1}, and we are done.

In situation (ii), we can assume that $z_p \ne p$, otherwise we are also in situation (i). Observe that the segment $l \subset L_U(p)\,$ joining $p$ and $z_p$ is a subset of $W^u(p)$, so it is contained in the attractor $\Lambda$.

Let $\tilde{z}_p= f^{-2d}(z_p) \in l$ and $\tilde{\Delta}(x,\varepsilon)=f^{-2d}(\Delta(x, \varepsilon))$, 
where $\pi(p)=d$ (see Figure \ref{cc}). Denote by $\tilde{l}$ the curve joining 
$z_p$ and $\tilde{z}_p$ inside $l$. Since the curve $\tilde{l}$ is a subset of $\Lambda$, it is accumulated by periodic points of $\Lambda$. 

\smallskip

If a periodic point $q$ is close to $\tilde{z}_p$, then $L_U(q)$ meets $\tilde{\Delta}(x, \varepsilon)$ transversely. We can also assume that $L_U(q)$ intersect $\triangle(x, \varepsilon)$, since $p$ can be chosen arbitrarily close to $x$. We take such point $q \in P_3$ lying between two points $z_q,\tilde{z}_q$ in $L_U(q)$ such that 
\begin{equation} \label{e.pitchfork}
z_q \in L_U(q)\pitchfork  \Delta(x,\varepsilon) \quad
\mbox{and}
\quad
\tilde{z}_q \in L_U(q)\pitchfork  \tilde{\Delta}(x,\varepsilon)\ \ \mbox{(see Figure~\ref{cc})}.
\end{equation}

\smallskip

If the point $z_q$ belongs to the stable manifold of some 
periodic point of index $d^s+1$, then, as in (i), Lemma~\ref{l.a} implies Equation~\eqref{e.case1}, and we are done. Thus, we can assume that
$z_q$ does not belong to the stable manifold of any point of
$L_U(q)$. As $q \in P_3$, this assumption implies that $z_{q}$ lies in the unstable manifold of the extremal point of $\Gamma_{q}$ of index $d^s$. Consequently,  $\tilde{z}_{q}$ lies in the stable manifold of some periodic point of $\Gamma_{q}$
of index $d^s+1$. 

\smallskip

By the coherent choice of the central curves and Equation~\ref{e.pitchfork}, the curve $L_U(f^{2d}(q)) \subset f^{2d}(L_U(q))$ meets  $\Delta(x, \varepsilon)$ at the point $f^{2d}(\tilde{z}_q)$. Clearly, this intersection lies in the stable manifold of some periodic point of $\Gamma_{f^{2d}(q)}$ of index $d^s+1$. Hence, we are in the same situation as in Lemma~\ref{l.a}, replacing $z_p$ and $p$ by $f^{2d}(\tilde{z}_q)$ and $f^{2d}(q)$, respectively. As in the lemma, this leads to equation~\eqref{e.case1}, ending item (ii).

\smallskip

\begin{figure}[hnt!] \centering
\includegraphics[scale=0.7]{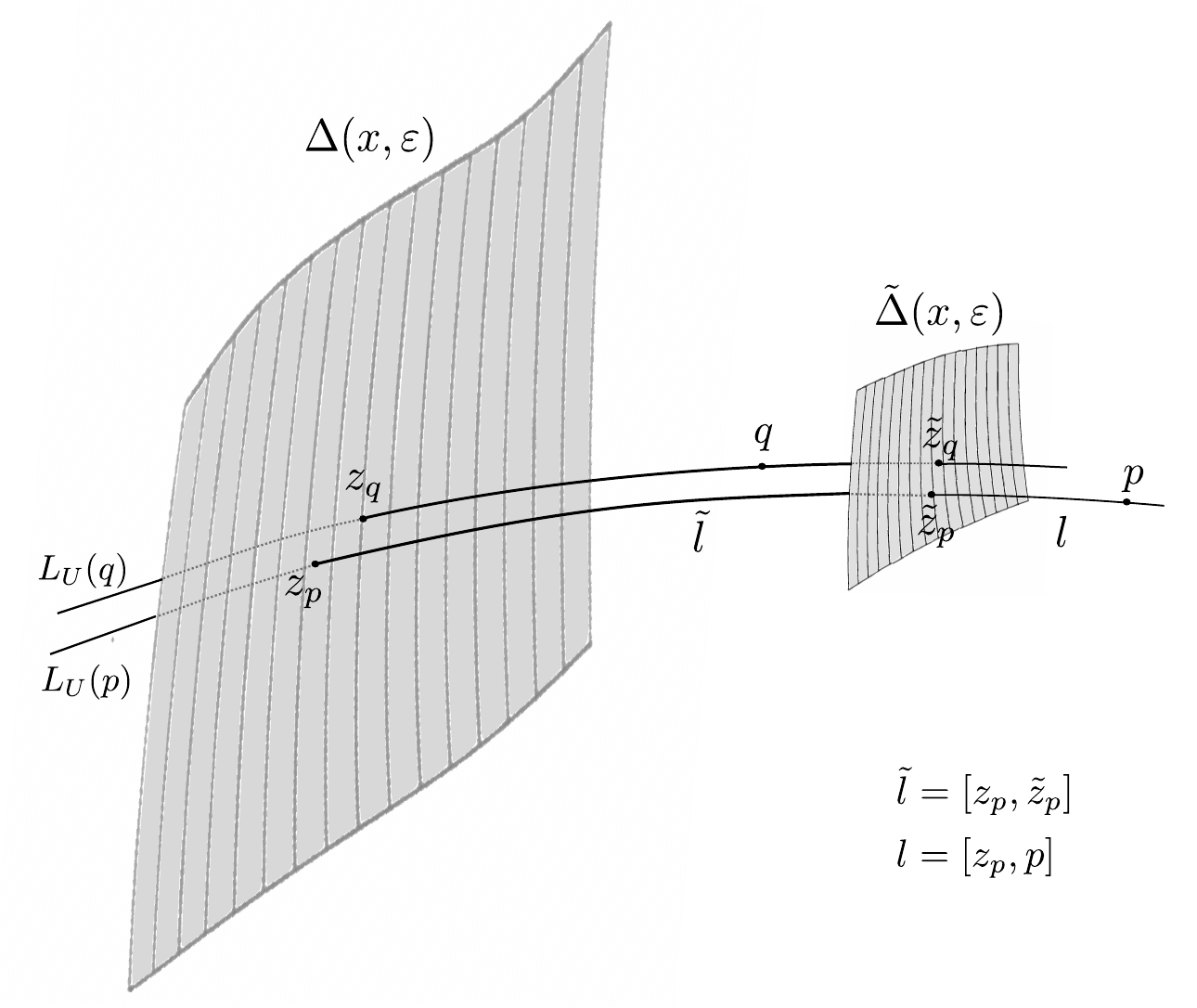}
\caption{Case (c)}
\label{cc}
\end{figure} 

We are left with the last possibility of item (iii). Let $a$ be a periodic point in $\triangle (x,\varepsilon)$. We can assume that $\varepsilon$ is sufficiently small so that, by Lemma \ref{l.per}, $a$ is the only periodic point in  $\triangle (a, \varepsilon)$. Then, we can choose $ \hat{x} \in \mathcal{F}^{u}_{\varepsilon}(a)$ and $\hat{\varepsilon} >0$ such that $\triangle(\hat{x}, \hat{\varepsilon})$ has no periodic points. It means that, replacing $x$ by $\hat{x}$ and $\varepsilon$ by $\hat{\varepsilon}$ in the begining of the proof, and following the same steps, situation (iii) do not occur. Arguing as before, we conclude the equivalent of Equation~\ref{e.case1} to the point $\hat{x}$, that is, $\overline{\mathcal{O}(\mathcal{F}^{u} (\hat{x}))} \cap 
\operatorname{Per}_{d^s+1}(f_{|_{\Lambda}}) \not= \emptyset.$ \\

As $\tilde{x} \in \mathcal{F}^{u}_{\varepsilon}(a)$, it implies that
\begin{equation}
\label{e.xxx} \overline{\mathcal{O}(\mathcal{F}^{u}(a))} \cap 
\operatorname{Per}_{d^s+1}(f_{|_{\Lambda}}) \not= \emptyset.
\end{equation} 

Finally, as $a \in \triangle (x, \varepsilon)$, there is $w \in \mathcal{F}_{\varepsilon}^u(x)$ such that $a \in \mathcal{F}_{\varepsilon}^s(w)$, so the orbit of $\mathcal{F}_{\varepsilon}^u(x)$ accumulates at the orbit of $\mathcal{F}_{\varepsilon}^u(a)$. This fact together with Equation~\ref{e.xxx} implies Equation~\ref{e.case1}. 

This completes the analysis of the three cases in Proposition~\ref{p.cases}.
\end{proof}

\section{Proof of Theorems A and B}
\bigskip

To get Theorem $A$ from Proposition \ref{p.cases}, it is left to  show that a generic $s$-minimal (resp. $u$-minimal) attractor is generically $s$-minimal (resp. $u$-minimal). Theorem $B$ follows from the fact that a generically $s$-minimal (resp. $u$-minimal) attractor that is robustly transitive is also robustly $s$-minimal (resp. $u$-minimal). 

\begin{lema} \label{big} 
Let $f \in \mathcal{R}$ and $\Lambda_f(U)$ be an $s$-minimal isolated partially hyperbolic set with one-dimensional center bundle. Let
 $\mathcal{U}$ be a compatible neighborhood of $f$.
For every hyperbolic periodic point $p\in \Lambda_f(U)$, 
there is an open set $\mathcal{W}_p \subset \mathcal{U}$, 
with $f \in \overline{\mathcal{W}_p}$, 
such that, for every $g \in \mathcal{W}_p$ and every strong stable disk $D$ centered at some point $x\in \Lambda_g(U)$, we have
$$
H(p_g,g) \subset \overline{\mathcal{O}^-_g(D)}.
$$
 Moreover, if $\operatorname{index}(p)=d^s$, then $\mathcal{W}_p$ is a neighborhood of $f$.

\end{lema}

\begin{re} \em{Lemma~\ref{big} has a dual version for $u$-minimal sets and the forward orbit $\overline{\mathcal{O}_g^+(D)}$ of an unstable disk $D$. The proof is analogous.} \end{re}  

\begin{proof}[Proof of Lemma~\ref{big}]

First we consider the case where $p \in \Lambda_f(U)$ has index~$d^s$.

Remark~\ref{cont.isolated} implies the upper semicontinuity of the sets $\Lambda_g(U)$, so   we can assume that every diffeomorphism in $\mathcal{R}$ is a continuity point of the map $g \mapsto\Lambda_g(U)$.
  
Consider the local unstable manifold  
$W^u_{\varepsilon}(\mathcal{O}_f(p))$.
By $s$-minimality, for each $x \in \Lambda_f(U)$ the leaf $\mathcal{F}^{s}(x)$ 
intersects $W^u_{\varepsilon}(\mathcal{O}_f(p))$ transversely (see Remark \ref{r.intersecta.unstable}). 
Hence, by the continuity\footnote{Here we are using the two kinds of continuity that $\mathcal{F}^s$ admits. First, the continuous dependence of $\mathcal{F}^s(g)$ on $g$ (see Remark~\ref{r.continuity}), and secondly the continuous dependence of $\mathcal{F}^s_r(x,g)$ on $x \in \Lambda_g(U)$.} of the strong stable lamination, there are open neighborhoods $U_x$ of $x$  
and $\mathcal{V}_x$ of $f$ such that, for every $y \in U_x\cap \Lambda_g(U)$ and every $g \in \mathcal{V}_x$, the leaf $\mathcal{F}^{s}(y,g)$  
intersects $W^u_{\varepsilon}(\mathcal{O}_g(p_g), g)$
transversely.

\smallskip

Covering the compact set $\Lambda_f(U)$ with the open sets $U_x $, we can extract a finite covering
$B = \bigcup_{i=1}^m U_{x_i}$ of $\Lambda_f(U)$. Let $\mathcal{V}$ denote the open set $\bigcap_{i=1}^m \mathcal{V}_{x_i}$.
Since $f$ is a continuity point of the map  
$g \mapsto\Lambda_g(U)$, after shrinking $\mathcal{V}$ if necessary, the inclusion 
$\Lambda_g(U) \subset B$ holds for every $g\in \mathcal{V}$.
By construction,
for every point 
$y \in \Lambda_g(U)$  and every 
$g \in \mathcal{V}$, 
the leaf $\mathcal{F}^{s}(y,g)$ intersects  $W^u_{\varepsilon}(\mathcal{O}(p_g), g)$ transversely. 

\smallskip

Applying the $\lambda$-Lemma to this transverse intersection we get\\

$W^s (\mathcal{O}_g (p_g), g) = \mathcal{O}_g(\mathcal{F}^s( p_g, g)) \subset
\overline{\mathcal{O}_g(\mathcal{F}^{s}(y,g))}$, and thus
\begin{equation} \label{e.hoo}
 H(p_g,g)\subset
\overline{ \mathcal{O}_g(\mathcal{F}^{s}(p_g , g))} 
\subset 
\overline{\mathcal{O}_g(\mathcal{F}^{s}(y,g))}.
\end{equation}

Note that if $D$ is any strong stable disk centered at some point $x \in \Lambda_g(U)$ and $y$ is an accumulation point of the pre-orbit of $x$, then 
\begin{equation}\label{e.hoo2}
\overline{\mathcal{O}_g(\mathcal{F}^{s}(y,g))} \subset \overline{\mathcal{O}^-_g(D)}.
\end{equation}

Combining Equations \eqref{e.hoo} and \eqref{e.hoo2}, we conclude the theorem for the case of $\operatorname{index}(p)=d^s$, where  $\mathcal{W}_p = \mathcal{V}$. Observe that in this case $\mathcal{W}_p$ is a neighborhood of $f$.

\smallskip

When $\operatorname{index}(p)=d^s+1$, consider another periodic point  $q \in \Lambda_f(U)$  with index $d^s$ (the existence of such point is assured by (2) of Proposition~\ref{Ab}). There is an open set $\mathcal{V}_{p,q}$ of $\mathcal{U}$ such that, if $g \in \mathcal{V}_{p,q}$, then\\

$H(p_g,g) \subset \overline{W^s(\mathcal{O}_g(p_g),g)} \subset   \overline{ \mathcal{O}_g(\mathcal{F}^{s}(q_g , g))}.$ (see Proposition \ref{Bld})\\

Applying the first part of the proof for $q$ (which has index $d^s$), we obtain a neighborhood  $\mathcal{W}_q$ of $f$ such that, for every $g \in \mathcal{W}_q$, it holds that 
$\overline{ \mathcal{O}_g(\mathcal{F}^{s}(q_g , g))} 
\subset \overline{\mathcal{O}^-_g(D)}$. By setting $\mathcal{W}_p = \mathcal{W}_q \cap \mathcal{V}_{p,q}$, we conclude that the inclusion $H(p_g,g)\subset \overline{\mathcal{O}^-_g(D)}$ holds for every $g \in \mathcal{W}_p$.

Recall that $f \in \overline{\mathcal{V}_{p,q}}$ and $\mathcal{W}_q$ is a neighborhood of $f$, so $f \in \overline{\mathcal{W}_p}$. 
\end{proof}

 Now we are ready to finish the proof of Theorems A and B.
 
\begin{proof}[Proof of Theorem A]  By proposition~\ref{p.cases}, there is a residual subset $\mathcal{G} \subset \mathcal{R}$ of $\operatorname{GTPHA}_1(U)$ such that, if $f \in \mathcal{G}$, then $\Lambda_f(U)$ is either $s$-minimal or $u$-minimal.

Suppose that $f \in  \mathcal{G}$ is such that $\Lambda_f(U)$ is $s$-minimal (the case where $\Lambda_f(U)$ is $u$-minimal goes similarly). Let $p \in \Lambda_f(U)$ be a periodic point of index $d^s$ (given by item (2) of Proposition~\ref{Ab}) and $\mathcal{W}_p$ be the neighborhood of $f$ given in Lemma~\ref{big}. For every $g \in \mathcal{W}_p$ and every disk $D= \mathcal{F}^s_r(x,g)$, with $x \in \Lambda_g(U)$, it holds that $ H(p_g,g) \subset \overline{\mathcal{O}^-_g(D)}.$

\smallskip

By Remark~\ref{5gta}, a transitive attractor is, generically, a homoclinic class. Hence, there is a residual subset $\mathcal{Z}$ of $\mathcal{W}_p$ such that, for every $g \in \mathcal{Z}$ and every disk $D= \mathcal{F}^s_r(x,g)$, with $x \in \Lambda_f(U)$, it holds that $\Lambda_g(U) \subset \overline{\mathcal{O}^-_g(D)}$.
In particular, it holds that $\Lambda_g(U) \subset \overline{\mathcal{O}(\mathcal{F}^{s}(x,g))} $ for every $x \in \Lambda_g(U)$. By Corollary~\ref{t.criterion}, $\Lambda_g(U)$ is $s$-minimal. Since it holds for every $g$ in $\mathcal{G} \cap \mathcal{W}_p$, the set $\Lambda_f(U)$ is generically $s$-minimal.
\end{proof}

\begin{proof}[Proof of Theorem B]

Let $\mathcal{B}$ be the subset of $\operatorname{RTPHA}_1(U)$ given by Corollary~\ref{CCCCC}. We can assume that, for every $f \in B$, the attractor $\Lambda_f(U)$ has  periodic points with index $d^s$ and $d^s+1$, as it holds generically (see Proposition~\ref{Ab}) and the existence of such points persist by small perturbations.  By Theorem A, there are two locally generic sets $\mathcal{G}_s$ and $\mathcal{G}_u$ such that, if $g \in \mathcal{G}_s$ then $\Lambda_g(U)$ is a generically $s$-minimal set, and if $g \in \mathcal{G}_u$ then $\Lambda_g(U)$ is a generically $u$-minimal set. 

Given $f \in \mathcal{B}\cap \mathcal{G}_s$, and $p \in \Lambda_f(U)$ with index $d^s$, we apply Lemma~\ref{big} to obtain a neighborhood $\mathcal{W}_p$ of $f$ satisfying the following. For every $g \in \mathcal{W}_p$ and every disk $D= \mathcal{F}^s_r(x,g)$, with $x \in \Lambda_g(U)$, it holds that $ H(p_g,g) \subset \overline{\mathcal{O}^-_g(D)}$.

Consider the open neighborhood $\mathcal{B}_f = \mathcal{W}_p \cap \mathcal{B}\, $ of $f$. By Corollary~\ref{CCCCC}, $\Lambda_g(U) = H(p_g,g)$ for every $g \in \mathcal{B}_f$. Hence, $\Lambda_g(U) \subset \overline{\mathcal{O}^-_g(D)}$ for every $x \in \Lambda_g(U)$. In particular, it holds that $\Lambda_g(U) \subset \overline{\mathcal{O}(\mathcal{F}^{s}(x,g))} $ for every $x \in \Lambda_g(U)$. By Remark~\ref{otimo}, Theorem~\ref{t.criterion2} can be applied in an open subset $\mathcal{C}_f$ of $\mathcal{B}_f$, so that $\Lambda_g(U)$ is an $s$-minimal set for every $g$ in $\mathcal{C}_f$.

Setting $\displaystyle \mathcal{B}_s = \bigcup_{f \in \mathcal{B}\cap \mathcal{G}_s} \mathcal{C}_f$, it gives an open subset of $\mathcal{B}$ such that, for every $g \in \mathcal{B}_s$, the attractor $\Lambda_g(U)$ is robustly $s$-minimal.  In the same way, we can obtain an open subset $\mathcal{B}_u$ of $\mathcal{B}$ such that, for every $g \in \mathcal{B}_u$, the attractor $\Lambda_g(U)$ is robustly $u$-minimal. By construction, the union $\mathcal{B}_s \cup \mathcal{B}_u$ is an open and dense subset of $\operatorname{RTPHA}_1(U)$. \end{proof}

\section{Final Considerations} \label{fc}

In this section we discuss the minimality in some examples and see why it is not possible to extend our results from the setting of robustly transitive attractors to the one of robustly transitive sets.
\smallskip

The main examples of non-hyperbolic robustly transitive diffeomorphisms are the following:

\begin{enumerate}

\item skew-products: there are two types. The Shub's example in \cite{B18} is derived from the product map $(A,B): \mathbb{T}^2 \times \mathbb{T}^2 \to \mathbb{T}^2 \times \mathbb{T}^2$ of two Anosov diffeomorphisms. The Bonatti-Díaz example in $\mathbb{T}^n \times N$, where $N$ is any compact manifold \cite{B19}.

\item DA diffeomorphism  (Mañé's example in in $\mathbb{T}^3$ \cite{B16}, and the generalizations of Bonatti-Viana in $\mathbb{T}^4$ \cite{B27} ).

\item perturbations of the time-1 map of Anosov flows, \cite{B19}.

\end{enumerate}

An easy way to produce examples of proper robustly transitive attractors is to consider the product map $F = f \times g$ of a robustly transitive diffeomorphism $f: M \to M$ with a north-south dynamics  $g : S^1 \to S^1$ in the circle.

 \smallskip

Denote by $\Lambda_F$ the attractor $M \times \{p\}$ of $F$, where $p$ is the hyperbolic attracting fixed point of $g$. If $\mathcal{F}^s$ (resp. $\mathcal{F}^u$) is minimal for $f$, then $\Lambda_F$ is robustly $s$-minimal (resp. $u$-minimal).

\smallskip

A different class of examples is obtained directly by skew-products as in (1) above. Instead of an Anosov system as the factor map of the skew-product, we consider a map $f: M \to M$  with a hyperbolic transitive attractor $\Lambda \subset M$. In \cite{B19} it is proved that the product of this map with the identity on $S^1$ can be perturbed to obtain  robustly non-hyperbolic transitive attractors homeomorphic to $\Lambda \times S^1$. These attractors have one-dimensional center bundle and a dense amount of compact central curves. By item (4) of Proposition~\ref{p.cases}, these attractors provide examples which are both $u$- and $s$-minimal.
\smallskip

As far as we know, the following question is still  open.

\begin{q} Is  there a robustly transitive diffeomorphism for which one of the minimality fails?  
\end{q}

If such $f$ do exist, then the construction above for $F = f \times g\ $ and $\tilde{F} = f^{-1} \times g\ $ gives examples of strictly $s$-minimal and striclty $u$-minimal attractors. By "strictly" we mean that the attractor has only that kind of minimality. Unfurtunately, every known example admits the two kinds of minimality. This lead us to pose the following question.

\begin{q} Is there a strictly $u$-minimal (resp. $s$-minimal) attractor?   
\end{q}

 Observe that if Question 2 has a negative answer, so has Question 1. 

\smallskip

Concerning possible generalizations of Theorems A and B, one may ask if these results could be obtained to the broader setting of robustly or generically transitive sets (instead of attractors). Next we show that, by an example in \cite{B20}, it is not possible (at least for the generic case).

\smallskip

In \cite{B20} it is proved that every manifold $M$ of dimension $\geq 3$ supports generically transitive sets that is not robustly transitive. The construction gives a diffeomorphism $f$ and a dense subset of a $C^1$ neighborhood of $f$ for which the semicontinuation $\Lambda_g(U)$ of the isolated set $\Lambda_f(U)$ has an isolated point (so it is not transitive). In addition, $\Lambda_f(U)$ is strongly partially hyperbolic with one-dimensional center bundle, which is the case we treat in this paper. 

\smallskip

In their construction, there are two hyperbolic periodic points  $p$ and $q$, with $\operatorname{index}(p) = \operatorname{index}(q)+1$, that lie in the ``corner'' of the set $\Lambda_f(U)$ (see the precise definition of \emph{cuspidal point} in \cite{B20}). These points have the property that $\mathcal{F}^s(p) \cap \Lambda_f(U) = \{p\}$ and $\mathcal{F}^s(q) \cap \Lambda_f(U) = \{q\}$, preventing both laminations to be minimal. Moreover, being a cuspidal point is a robust property, so the continuations $\Lambda_g(U)$ do not have minimal laminations either.

\smallskip

\section*{Acknowledgements}

This work is part of my PhD Thesis at PUC-Rio. I specially thank Flavio Abdenur for his commitment and support during my stay at PUC. I'm also very grateful to Lorenzo J. Díaz, who advise me on the last year of my doctoral program. I surely learned a lot from them.  I'm also grateful to the other professors, students, and the staff of the Department of Mathematics of PUC-Rio.

\end{document}